\documentclass{dcds}
\usepackage{amsmath}
  \usepackage{paralist}
  \usepackage{graphics} 
  \usepackage{epsfig} 
\usepackage{graphicx} \usepackage{epstopdf}
 \usepackage[colorlinks=true]{hyperref}
\hypersetup{urlcolor=blue, citecolor=red}

\def\currentvolume{39}
 \def\currentissue{5}
  \def\currentyear{2019}
   \def\currentmonth{May}
    \def\ppages{X--XX}
     \def\DOI{10.3934/dcds.2019103}

\newtheorem{Thm}{Theorem}[section]
\newtheorem{Cor}{Corollary}
\newtheorem{MyTh}{Main Theorem}
\newtheorem{Lem}[Thm]{Lemma}
\newtheorem{Prop}{Proposition}
\newtheorem{Ex}{Example}

\theoremstyle{definition}
\newtheorem{Def}[Thm]{Definition}
\newtheorem{Rem}{Remark}

\title[Construction of Lyapunov Functions Using HHD] 
      {Construction of Lyapunov functions using Helmholtz--Hodge decomposition}

\author[Tomoharu Suda]{}

\subjclass{Primary: 37B25; Secondary: 37C10, 37B35.}
 \keywords{Helmholtz-Hodge decomposition, Lyapunov functions, gradient vector fields, strict orthogonality, potential functions.}

\email{suda.tomoharu.88s@st.kyoto-u.ac.jp}
\thanks{The first author is supported by Grant-in-Aid for JSPS Fellows (17J03931).}

\usepackage[pagewise]{lineno}
\usepackage{comment,here}

\begin{document}
\maketitle

\centerline{\scshape Tomoharu Suda}
\medskip
{\footnotesize
 \centerline{Graduate School of Human and Environmental Studies, Kyoto University}
   \centerline{Yoshida-nihonmatsu-cho, Sakyo-ku, Kyoto, 606-8501, Japan}
} 

\bigskip
 \centerline{(Communicated by Bernold Fiedler)}
\begin{abstract} The Helmholtz--Hodge decomposition (HHD) is applied to the construction of Lyapunov functions. It is shown that if a stability condition is satisfied, such a decomposition can be chosen so that its potential function is a Lyapunov function. In connection with the Lyapunov function, vector fields with strictly orthogonal HHD are analyzed. It is shown that they are a generalization of  gradient vector fields and have similar properties. Finally, to examine the limitations of the proposed method, planar vector fields are analyzed.
\end{abstract}

\section{Introduction}

The Helmholtz--Hodge decomposition (HHD) is a decomposition of vector fields whereby they are expressed as the sum of a gradient vector field and a divergence-free vector field. It is applied to the study of Navier--Stokes equations (projection method) and to the detection and visualization of singularities of vector fields. For a review of the literature on HHD, the reader may refer to \cite{Bhatia_The_2013}. There are several methods for obtaining HHD. The most naive is by solving the Poisson equation, as its solution yields such a decomposition \cite{Geng_The_2010}. Thus, an HHD exists, provided that the Poisson equation has a solution. However, a certain boundary condition, whose choice is not obvious, must be prescribed. This condition is often imposed so that the two vector fields in the decomposition are orthogonal in the $L^2$ sense, and it is discussed in \cite{Denaro_On_2003}. Other techniques to obtain an HHD involve the construction of basis functions or Green's function \cite{ Bhatia_The_2014,Fuselier_A_2016} and thus avoid the complications related to the choice of boundary conditions.

Although there are examples of HHD applied to the detection of singularities of vector fields \cite{Polthier_Identifying_2003,Wiebel_Feature}, few studies have been concerned with HHD from the viewpoint of dynamical systems. In the planar case, Demongeot, Glade, and Forest proposed a scheme for the analysis of planar vector fields with a limit cycle, in which they constructed a polynomial vector field so that it has a limit cycle resembling that of the original vector field \cite{Demongeot_Li_2007, Demongeot_Li_2007_2}. Mendes and Duarte approximated $n$-dimensional vector fields by divergence-free vector fields \cite{Duarte_Deformation_1983}. Even though these methods are different, their aim is to determine a function that best captures the behavior of the given vector field using HHD.

To describe vector fields using a function, a natural option is to consider a Lyapunov function. A Lyapunov function of a vector field is a smooth function that takes minimum value at an equilibrium point  (more generally, on an invariant set), and decreases along solution curves in a neighborhood of the equilibrium point. The potential function of gradient vector fields is an example of Lyapunov function. Even though its existence is known for asymptotically stable equilibrium points of vector fields, the actual construction is not obvious except in the case of relatively simple vector fields. Several methods have been proposed for numerically constructing Lyapunov functions, and a comprehensive review can be found in \cite{Giesl_Review_2015}. They vary depending on the nature of the vector fields. In particular, methods based on linear programming are often applied to linear systems; however they may not be used for general vector fields. For an example of this construction, the reader is referred to \cite{599986}. If the vector fields are nonlinear, other techniques may be applied. Lyapunov functions can be constructed by considering a partial differential equation whose solution yields a Lyapunov function and then approximating this solution. A detailed discussion on such a method can be found in \cite{giesl2007construction}. This approach can handle a variety of vector fields; however, it cannot readily be associated with the resulting function in terms of dynamics.

From the viewpoint of dynamical systems, it appears natural to choose an HHD whose gradient vector field is generated by a Lyapunov function. Indeed, the method of Demongeot, Glade, and Forest can be interpreted as an attempt to choose an HHD whose potential function is a Lyapunov function of a limit cycle \cite{Demongeot_Li_2007_2}. Thus, it appears possible that a ``good'' choice of HHD yields a Lyapunov function. Furthermore, a Lyapunov function thereby obtained is naturally interpreted as the dissipative part of the original vector field. Hence, it is likely to offer insight into its dynamics.

In this study, the construction of Lyapunov functions using HHD is considered. It is shown that, under certain stability conditions, this is possible, and a method is provided for formulating it as an optimization problem. Moreover, strictly orthogonal HHDs are considered. This notion of orthogonality differs from the $L^2$-orthogonality commonly imposed in the literature meaning that the two vector fields obtained by the decomposition are orthogonal everywhere. It is proved that vector fields with strictly orthogonal HHD are similar to gradient vector fields and their behavior is completely determined in terms of potential functions that are Lyapunov functions of invariant sets. Finally, the limitations of the proposed method are examined in the case of planar vector fields.

\medskip
\noindent{\bf Main Results.}
\noindent The terminology used in this study is first provided. It is assumed that all vector fields are  $C^2$, unless indicated otherwise. Vectors will be denoted by bold fonts.  Autonomous differential equations are considered and will be given in the form
\begin{equation}\label{diffeq}
\dot {\bf x} = {\bf F}({\bf x}),
\end{equation}
 where ${\bf F}({\bf x})$ is a smooth vector field defined on $\bar \Omega$, where $\Omega \subset {\mathbb R}^n$ is an open domain which contains the origin. In what follows, unless otherwise stated, it is assumed that the domain $\Omega$ is bounded and has $C^2$-boundary.
\begin{Def}[Orbital derivative]
\normalfont
For a $C^1$-function $f$ on a domain $\Omega \subset {\mathbb R}^n$, its {\bf orbital derivative} $\dot{f}$ along solutions of  the equation (\ref{diffeq}) is a function defined by
$$\dot{f}({\bf x}(t)) :=  \frac{d}{dt} f\left({\bf x}(t)\right) = \sum_{i}\frac{\partial f}{\partial x_i}\frac{d x_i(t)}{d t},$$
where ${\bf x}(t) $ is a solution of the equation (\ref{diffeq}). Since the vector field is autonomous, the orbital derivative $\dot{f}$ is given explicitly as a function of ${\bf x}$, namely,
$$\dot{f}({\bf x}) = \nabla f ({\bf x}) \cdot {\bf F}({\bf x}). $$
\end{Def}

Lyapunov functions are defined as follows, requiring that they take a minimum value at an equilibrium point.
\begin{Def}[Lyapunov functions]
		Let ${\bf x}_*$ be an equilibrium point and $U \subset {\mathbb R}^n$ a neighborhood of ${\bf x}_*$. A $C^1$-function $L:U\rightarrow{\mathbb R}$ is said to be a {\bf (local) Lyapunov function} if the following hold:
		\begin{enumerate}
			\item For all ${\bf x}\in U- \{{\bf x}_*\}$, we have $L({\bf x})> L({\bf x}_*).$
			\item On the set $U- \{{\bf x}_*\}$, we have ${\dot L}\leq 0.$
		\end{enumerate}
\end{Def}
In this study, the Helmholtz--Hodge decomposition is defined as follows.
\begin{Def}[Helmholtz--Hodge decomposition]
	\normalfont
		For a vector field $ {\bf F}$ on a domain $\Omega \subset {\mathbb R}^n$, its {\bf Helmholtz--Hodge decomposition} ({\bf HHD}) is a decomposition of the form
	$$ {\bf F} = -\nabla V + {\bf u},$$
where $V: \Omega \rightarrow\mathbb{R}$ is a $C^3$ function and ${\bf u}$ is a vector field on $\Omega$ with the property $\nabla \cdot {\bf u} = 0$. $V$ is called a {\bf potential function}.
\end{Def}
From the definition, it can be seen that the addition of a harmonic function $h$ to the potential function $V$ yields another decomposition, $ {\bf F} = -\nabla\left( V + h\right) + \left({\bf u} + \nabla h \right)$. Therefore HHD is not unique in general. The existence of HHD is established if there is a solution of the Poisson equation obtained by taking the divergence of the equation in the definition.  Such a solution exists if the boundary of the domain under consideration is, for example, sufficiently smooth. More details may be found in \cite{Geng_The_2010}.  By the assumption on the smoothness of $\partial \Omega$, at least one HHD exists for the vector field considered on $\Omega$.

In the analysis of potential functions, the following concept is used.
\begin{Def}[Singular values of a matrix]
Let $A$ be a square real matrix of order $n$. Then, the matrix $\,^{t} A A$ is nonnegative definite; therefore, it has positive eigenvalues. The square root of an eigenvalue of  $\,^{t} A A$ is called a {\bf singular value} of $A$.
\end{Def}
\noindent  The singular values of a matrix are denoted by Greek letters, for example, $\lambda$ or $\mu$. It is easily verified that a matrix and its transpose have the same set of singular values. More details may be found in \cite{strang2006linear}.

The main results of this study will now be stated. It is first shown that under a certain stability condition, it is possible to obtain an HHD with its potential function assuming a minimum at an equilibrium point.

\begin{MyTh}\label{Const_Potential}\label{Lya}
Let  ${\bf F}:\bar \Omega \to {\mathbb R}^n$ be a vector field with an equilibrium point at the origin. If $\nabla \cdot {\bf F}({\bf 0}) < 0$, then it has an HHD with the potential function $V$ having a minimum at the origin.
\end{MyTh}
To obtain a Lyapunov function,  a function whose orbital derivative decreases monotonically should be determined. This may be achieved by controlling the Jacobian matrix of the divergence-free vector field ${\bf u}$ and the Hessian of the potential function  $V$. The next theorem provides a sufficient condition whereby the potential function is a Lyapunov function in a form suitable for optimization. The Jacobian matrix of ${\bf u}$ and ${\bf F}$ at the origin is denoted by $D {\bf u}_{\bf 0}$ and $D {\bf F}_{\bf 0}$, respectively.
\begin{MyTh}\label{opti_thm}
	Let ${\bf F}: \bar \Omega \to {\mathbb R}^n$ be a vector field with an equilibrium point at the origin and ${\bf F} = - \nabla V + {\bf u}$ be an HHD, assume $V$ has a minimum at the origin, and let the largest singular value of $D {\bf u}_{\bf 0}$ be $\lambda _{\bf u}$, the smallest singular value of  $D {\bf F}_{\bf 0}$ be $\mu _{\bf F}$, and the smallest eigenvalue of ${\rm Hess} \,V$ be $\mu_V$. If $ \lambda _{\bf u} ^2 - \mu_V ^2 < \mu _{\bf F} ^2$, then $V$ is a (local) Lyapunov function of the origin.
\end{MyTh}
A simple example of the use of Main Theorem \ref{opti_thm} will be given in Example \ref{Lya}.

In connection with Lyapunov functions, vector fields with strictly orthogonal HHD are an example where the analysis based on potential function succeeds.
	\begin{Def}
	HHD ${\bf F} =- \nabla V+ {\bf u} $ is said to be {\bf strictly orthogonal} on $D \subset {\mathbb R}^n$ if ${\bf u}({\bf x}) \cdot  \nabla V({\bf x})= 0$ for all ${\bf x}\in D$.
	\end{Def}
Although vector fields with strictly orthogonal HHD  exhibit behavior different from that of gradient vector fields, they are in fact a generalization of them, as the next theorem states.
\begin{MyTh}\label{ohhd}
	Let $D \subset {\mathbb R}^n$ be a bounded domain. If a vector field ${\bf F}:{\mathbb R}^n \to {\mathbb R}^n$ has a strictly orthogonal HHD on $\bar D$ with potential function $V$, then the following hold:
		\begin{enumerate}
		\item  If $D$ is forward invariant, then for all ${\bf x} \in D$, $$\omega({\bf x}) \subset \{{\bf y} \in {\bar D}\,| \,\nabla V({\bf y}) = {\bf 0}\}.$$
		\item If $D$ is backward invariant, then for all ${\bf x} \in D$,$$\alpha({\bf x}) \subset \{{\bf y} \in {\bar D}\,| \,\nabla V({\bf y}) = {\bf 0}\}.$$
		\end{enumerate}
	\end{MyTh}

This paper is organized as follows. In Section 2, it is shown that it is possible to construct Lyapunov functions by properly choosing an HHD. Moreover, Main Theorems \ref{Lya} and  \ref{opti_thm} are proved. In Section 3, the properties of strictly orthogonal HHDs are considered. Main Theorem \ref{ohhd} is proved, and its compatibility with a boundary condition widely used in literature is further studied. In Section 4,  to discuss the limitations of the analysis based on HHD, planar vector fields are analyzed using Fourier series expansions.

\section{Construction of Lyapunov functions using HHD}
In this section, the construction of Lyapunov functions near an equilibrium point is considered.  The emphasis here is not on solving actual problems, but on showing that it is possible to construct Lyapunov functions using HHD. Without loss of generality, it may be assumed that the equilibrium point is the origin.

The aim here is to choose the potential functions of an HHD so that they can be used as Lyapunov functions. To achieve this, it is first ensured that the potential function assumes a minimum value at the equilibrium point.
\begin{Thm}[Main Theorem \ref{Lya}]\label{Const_Potential}
Let  ${\bf F}: \bar \Omega \to {\mathbb R}^n$ be a vector field with an equilibrium point at the origin. If $\nabla \cdot {\bf F}({\bf 0}) < 0$, then it has an HHD such that the potential function $V$ has a minimum at the origin.
\end{Thm}
\begin{proof}
An HHD of ${\bf F}$ is first chosen, that is,
	$${\bf F} = -\nabla \tilde V + \tilde {\bf u}.$$
This is possible because an HHD exists in $\Omega$. If $\nabla \tilde V ({\bf 0}) \neq {\bf 0}$, then $\tilde V({\bf x})$ is replaced by $\tilde V({\bf x}) - \nabla \tilde V ({\bf 0}) \cdot {\bf x}$, which also yields an HHD. It should be noted that $\nabla \tilde V ({\bf 0})$ is merely a constant vector. Therefore, it may always be assumed that $\nabla \tilde V ({\bf 0}) = {\bf 0}$.

The Hessian of $V$ is denoted by ${\rm Hess}\, V$. As
	\begin{eqnarray*}
		\nabla \cdot {\bf F} &=& - \nabla \cdot \nabla \tilde V \\
							&=& - \sum_{i=1}^{n} \frac{\partial^2 \tilde V}{\partial x_{i} ^ 2}\\
								&=& - {\rm tr}\, {\rm Hess} \,\tilde V,
	\end{eqnarray*}
then ${\rm tr}\, {\rm Hess}\, \tilde V$ is positive at the origin. Therefore, the matrix $\frac{\rho}{n} I$ is positive definite, where $\rho$ is the value of ${\rm tr}\, {\rm Hess}\, \tilde V$ at the origin. Let the symmetric matrix $A$ be defined by $- {\rm Hess} \,\tilde V +\frac{\rho}{n} I $. Then, ${\rm tr}\, A = 0$, and $\,^{t}{\bf x} A {\bf x}\,$ is a harmonic function.

Let $V$ be defined by $\tilde V + \,^{t}{\bf x} A {\bf x}$. This function also yields an HHD ${\bf F} = -\nabla V + {\bf u}$, because $\,^{t}{\bf x} A {\bf x}$ is a harmonic function. As $\nabla V ({\bf 0}) = {\bf 0}$ and the Hessian of $V$ at the origin is
	\begin{eqnarray*}
		{\rm Hess}\, V &=& {\rm Hess}\, \tilde V + A\\
						&=& \frac{\rho}{n} I,
	\end{eqnarray*}
	we conclude that $V$ takes a minimum value at the origin, owing to diagonality and positive definiteness.
\end{proof}

\begin{Rem}
 The condition in the hypothesis of Main Theorem \ref{Lya} is related to the Jacobian matrix of the vector field at the origin via the identity  $\nabla \cdot {\bf F}({\bf 0}) = {\rm tr} D {\bf F}_{\bf 0}$.
If this hypothesis fails, there are two possibilities: $\nabla \cdot {\bf F}({\bf 0})>0$ or $\nabla \cdot {\bf F}({\bf 0})=0$.  In the former case, the origin cannot be a stable equilibrium point, and there is no Lyapunov function for it. In the latter case, the origin is a saddle, a center, or a degenerate equilibrium.
\end{Rem}
Even though the condition that $V$ takes a minimum value at the origin does not imply that it is a Lyapunov function, we have the following result.

\begin{Prop}\label{min_p}
	Let  ${\bf F}: \bar \Omega \to {\mathbb R}^n$ be a vector field with an equilibrium point at the origin, and let ${\bf F} = - \nabla V + {\bf u}$ be an HHD. If $V$ assumes a minimum at the origin, then $\dot V ({\bf 0}) = 0$ and the origin is a critical point of $\dot V$. The Hessian of $\dot V$ is given by $ \,^{t} \left(D {\bf F}_{\bf 0} \right)\left({\rm Hess} \, V \right) + \left({\rm Hess} \, V \right) D {\bf F}_{\bf 0}$.
\end{Prop}

\begin{proof}
	By the definition of the orbital derivative, we have $\dot V = \nabla V \cdot {\bf F}= - \left|\nabla V \right|^2 + \nabla V \cdot {\bf u} $. As $V$ assumes a minimum at the origin,  $\dot V ({\bf 0})$ is zero.

The first-order derivatives of $\dot V$ are
	\begin{eqnarray*}
		\frac{\partial \dot V}{\partial x_{i}} &=& \frac{\partial}{\partial x_{i}} \sum_{l = 1}^n \left( -\frac{\partial V}{\partial x_{l}} \frac{\partial V}{\partial x_{l}} + \frac{\partial V}{\partial x_{l}} u_l \right)\\
				&=& -2 \sum_{l = 1}^n\frac{\partial^2 V}{\partial x_{l} \partial x_{i}} \frac{\partial V}{\partial x_{l}} + \sum_{l = 1}^n\frac{\partial^2 V}{\partial x_{l} \partial x_{i}} u_l +\sum_{l = 1}^n\frac{\partial V}{\partial x_{l}} \frac{\partial u_l}{\partial x_i}.
	\end{eqnarray*}
The first and last terms vanish at the origin because $V$ assumes a minimum. As ${\bf u}({\bf 0}) = {\bf F}({\bf 0}) + \nabla V ({\bf 0}) = 0$, the second term also vanishes. Therefore, the origin is a critical point of $\dot V$.

The second-order derivatives of $\dot V$ at the origin are
$$
	\frac{\partial^2 \dot V}{\partial x_i \partial x_j} = -2 \sum_{l = 1}^n\frac{\partial^2 V}{\partial x_{l} \partial x_{i}} \frac{\partial^2 V}{\partial x_{l} \partial x_j} + \sum_{l = 1}^n\frac{\partial^2 V}{\partial x_{l} \partial x_{j}} \frac{\partial u_l}{\partial x_i} +\sum_{l = 1}^n\frac{\partial^2 V}{\partial x_{l} \partial x_i} \frac{\partial u_l}{\partial x_j}.
$$
Therefore,
$$ {\rm Hess} \, \dot V = -2  \left({\rm Hess} \, V \right)\left(  {\rm Hess} \, V \right) +  \,^t \left(D {\bf u}_{\bf 0}\right) \left({\rm Hess} \, V \right) + \left({\rm Hess} \, V \right) D {\bf u}_{\bf 0}. $$
As $D {\bf F}_{\bf 0} = -  \left({\rm Hess} \, V \right) +D {\bf u}_{\bf 0}$, this is equivalent to $ \, ^t \left(D {\bf F}_{\bf 0} \right)\left({\rm Hess} \, V \right) + \left({\rm Hess} \, V \right) D {\bf F}_{\bf 0}$.
\end{proof}
To summarize, an HHD of a vector field ${\bf F}$ with an equilibrium point at the origin yields a Lyapunov function if its potential function $V$ assumes a minimum at the origin and the matrix $ \,^t D {\bf F}_{\bf 0} \left({\rm Hess} \, V \right) + \left({\rm Hess} \, V \right) D {\bf F}_{\bf 0} $ is negative definite. If these hypotheses are satisfied, $\dot V$ takes a maximum value of $0$ at the origin; therefore, $\dot V \leq 0$ in a neighborhood of the origin. For example, we have the following result.

\begin{Cor}
If $\left<D {\bf F} _0 {\bf x},{\bf x}\right>\,< 0$ for all ${\bf x} \in {\mathbb R}^n \backslash \{{\bf 0}\}$ in addition to the conditions in Main Theorem \ref{Lya}, the potential function $V$ constructed in Main Theorem \ref{Lya} is a Lyapunov function.
\end{Cor}
\begin{proof}
	As the potential function $V$ constructed in Main Theorem \ref{Lya} assumes a minimum at the origin, Proposition \ref{min_p} can be applied. We need only show that $\,^t D {\bf F}_{\bf 0} \left({\rm Hess} \, V \right) + \left({\rm Hess} \, V \right) D {\bf F}_{\bf 0} $ is negative definite.

As the Hessian of the potential function $V$ constructed in Main Theorem \ref{Lya} is  $\frac{\rho}{n} I,$ we have
	$$
		 \left<\left( \,^t D {\bf F}_{\bf 0} \left({\rm Hess} \, V \right) + \left({\rm Hess} \, V \right) D {\bf F}_{\bf 0} \right){\bf x},{\bf x}\right> = \frac{2 \rho}{n}\left <D {\bf F} _0 {\bf x},{\bf x}\right>\,< 0
	$$
if ${\bf x} \neq 0$.
\end{proof}
\begin{Rem}
\normalfont
 The condition that $\left<D {\bf F} _0 {\bf x},{\bf x}\right>\,< 0$ for all ${\bf x} \in {\mathbb R}^n \backslash \{{\bf 0}\}$ is equivalent to the symmetric matrix $D {\bf F} _0 + \,^t D {\bf F} _0 $ being negative definite. Therefore, the stability of the origin is easily obtained without resort to Lyapunov functions.
\end{Rem}
Even though a potential function that assumes a minimum at the origin can be chosen under a relatively mild condition, it is not obvious that $ \,^t D {\bf F}_{\bf 0} \left({\rm Hess} \, V \right) + \left({\rm Hess} \, V \right) D {\bf F}_{\bf 0} $ can be chosen so as to be negative definite. Thus, this problem is reformulated as an optimization problem, as shown by the next theorem.
\begin{Thm}[Main Theorem \ref{opti_thm}]
Let ${\bf F}: \bar \Omega \to {\mathbb R}^n$ be a vector field with an equilibrium point at the origin and ${\bf F} = - \nabla V + {\bf u}$ be an HHD, assume $V$ has a minimum at the origin, and let the largest singular value of $D {\bf u}_{\bf 0}$ be $\lambda _{\bf u}$, the smallest singular value of  $D {\bf F}_{\bf 0}$ be $\mu _{\bf F}$, and the smallest eigenvalue of ${\rm Hess} \,V$ be $\mu_V$. If $ \lambda _{\bf u} ^2 - \mu_V ^2 < \mu _{\bf F} ^2$, then $V$ is a (local) Lyapunov function of the origin.
\end{Thm}

\begin{proof}
 The proof is by direct calculation. Using the representation obtained in Proposition \ref{min_p},  we have
\begin{align*}
\left<  ({\rm Hess}\, \dot V)  {\bf x}, {\bf x} \right> &=\left<\left( \,^t D {\bf F}_{\bf 0} \left({\rm Hess} \, V \right) + \left({\rm Hess} \, V \right) D {\bf F}_{\bf 0} \right){\bf x},{\bf x}\right>\\
			&=-2 || ({\rm Hess}\,  V)  {\bf x}||^2 + 2 \left< D {\bf u}_{\bf 0} {\bf x}, \left({\rm Hess} \, V \right) {\bf x} \right>,\\
\intertext{since $D {\bf F}_{\bf 0} = D {\bf u}_{\bf 0} - {\rm Hess}\,  V$. Rewriting the inner product of the second term in the last line, we have}
			\left<  ({\rm Hess}\, \dot V)  {\bf x}, {\bf x} \right>&=- || ({\rm Hess} \,  V)  {\bf x}||^2 + ||D {\bf u}_{\bf 0} {\bf x} ||^2 - ||\left(D {\bf u}_{\bf 0}-{\rm Hess} \, V \right){\bf x}||^2. \\
\intertext{Estimating each term by singular values or eigenvalues, we obtain}
			\left<  ({\rm Hess}\, \dot V)  {\bf x}, {\bf x} \right> &=-|| ({\rm Hess} \,  V)  {\bf x}||^2 +
||D {\bf u}_{\bf 0} {\bf x} ||^2-||D {\bf F}_{\bf 0} {\bf x} ||^2\\
		&\leq \left( \lambda_{\bf u} ^2 - \mu_{\bf F} ^2 - \mu_{V}^2 \right) ||{\bf x}||^2.
\end{align*}
Therefore, $({\rm Hess}\, \dot V)$ is negative definite. Combining this with Proposition \ref{min_p}, we conclude that  ${\dot V}$ assumes a maximum value $0$ at the origin and $V$ is a Lyapunov function.
\end{proof}
Using Main Theorem \ref{opti_thm}, a Lyapunov function may be constructed by seeking a decomposition that satisfies $ \lambda _{\bf u} ^2 - \mu_V ^2 < \mu _{\bf F} ^2$, even though a solution does not necessarily exist. In particular, this method is not applicable to vector fields with degenerate Jacobian matrix $D {\bf F}_{\bf 0} $.

For optimization, a harmonic function $h$ may be constructed so that the condition $ \lambda _{{\bf u} + \nabla h} ^2 - \mu_{V+h} ^2 < \mu _{\bf F} ^2$ is satisfied. This could be carried out by, for example, considering harmonic polynomials or harmonic functions constructed using the Poisson kernel representation. As only the quadratic terms of harmonic functions are relevant to the calculation of $\lambda _{{\bf u} + \nabla h}$ and $\mu_{V+h}$, it will suffice to consider the addition of quadratic harmonic polynomials. However, the implementation of such a procedure requires further investigation.

\begin{Ex}\label{Lya_ex}
\normalfont
A simple example of the use of Main Theorem \ref{opti_thm} is now provided. Let the vector field ${\bf F}$ on ${\mathbb R}^2$ be given by
$$ {\bf F}(x,y) = \left(\begin{array}{cc}
									 	-a & b\\
										c & - b
									\end{array}
								\right) \left( \begin{array}{c} x \\ y \end{array} \right).$$
with $ a  > b > c >0$. Then, there is an HHD of ${\bf F} = - \nabla V  + {\bf u}$ given by
\begin{eqnarray*}
V &=& \frac{a}{2} x^2 + \frac{b}{2} y^2,\\
{\bf u} &=& \left(\begin{array}{c}
				b y \\
				c x \end{array} \right).
\end{eqnarray*}
This decomposition is obtained by selecting the ``diagonal terms'' of the vector field as $- \nabla V$ and the remaining terms as ${\bf u}$.
It is now shown that $V$ is a Lyapunov function of the origin. This is easily verified by direct calculation as well. As $\lambda_{\bf u} = \mu_V = b$, it suffices to show that $\mu_{\bf F}^2 > 0$. This is equivalent to the matrix $\, ^ t \left( D {\bf F}_{\bf 0} \right) \left( D {\bf F}_{\bf 0} \right)$ being regular, which is obviously true because the coefficient matrix is regular. By Main Theorem \ref{opti_thm}, we conclude that $V$ is a Lyapunov function of the origin.

This analysis can be extended to vector fields involving higher-order terms with the same linear part if they have a decomposition with the lower-order terms of the potential function and the divergence-free vector field in the form given here. For example, let a vector field ${\bf F}$ be given by
$$
 {\bf F} = \left(\begin{array}{cc}
									 	-a & b\\
										c & - b
									\end{array}
								\right) {\bf x}
-\nabla p
+\left(\begin{array}{c}
									 	q\\
										r
									\end{array}
								\right),
$$
where $p,q,r$ are homogeneous polynomials with ${\rm deg} (p) > 2$, ${\rm deg} (q), {\rm deg} (r) > 1$, and $\frac{\partial q}{\partial x} + \frac{\partial r}{\partial y} = 0$. Then, $V = \frac{a}{2} x^2 + \frac{b}{2} y^2 + p(x,y)$ is a Lyapunov function of the origin.
\end{Ex}

\begin{Ex}
\normalfont
Another example of the use of Main Theorem \ref{opti_thm} is now presented to illustrate the features of the proposed method, compared to other schemes. The following ordinary differential equation is considered:
\begin{equation}\label{Giesl_ex}
\frac{d}{dt}\left(\begin{array}{c}
										x\\
										y
									\end{array}
								\right)
=\left(\begin{array}{c}
										x\left(-1 + 4 x^2 + \frac{1}{4} y^2 \right) + \frac{1}{8}y^3\\
										y\left(-1+\frac{5}{2}x^2 + \frac{3}{8}y^2 \right)-6 x^3
									\end{array}
								\right).
\end{equation}
This equation was analyzed as an example of numerical methods for constructing Lyapunov functions in \cite{giesl2007construction}. Samples of solution curves are given in the right panel of Figure \ref{giesl_lya}. As the origin is an asymptotically stable equilibrium point of the vector field, a local Lyapunov function of the origin may be constructed. For example, linear approximation of the vector field yields a local Lyapunov function $V_0 = \frac{1}{2}\left(x^2 + y^2\right)$, which is strictly decreasing on the disk $\{(x,y) \in {\mathbb R}^2 \mid V_0(x,y) \leq 0.09\}$. For the details of the construction via linear approximation, the reader is referred to \cite{giesl2007construction}.

Local Lyapunov functions are now constructed by applying Main Theorem \ref{opti_thm}. A potential function $V$ of an HHD for the vector field in the equation (\ref{Giesl_ex}) satisfies the following Poisson equation:
\begin{equation}\label{g_eq}
\Delta V = 2-\frac{29}{2}x^2-\frac{11}{8} y^2,
\end{equation}
which is obtained by taking the divergence in the definition of HHD. As $-\nabla V_0$ gives the linear part of the vector field, it is natural to seek a potential function $V$ with the terms of order less than two given by $V_0$, so that the linear approximation of $-\nabla V$ coincides with $-\nabla V_0$. If such $V$ is chosen, it takes minimum value at the origin. Furthermore, we have $D {\bf u}_{\bf 0} = 0$; therefore, $\lambda_{\bf u} = 0$ and the hypotheses of Main Theorem \ref{opti_thm} hold trivially.

As no cross term of $x$ and $y$ appears in Equation (\ref{g_eq}), an obvious solution is given by integrating each non-constant term twice, namely,
$$V_1(x,y) = V_0(x,y) - \frac{29}{24}x^4 -\frac{11}{96}y^4.$$
Contours of $V_1$ and the sign of $\dot{V_1}$ are shown in the left panel of Figure \ref{giesl_lya}. Compared with $V_0$, it succeeds in expanding the domain where solutions are determined to be divergent, whereas it fails to improve the estimate for the local attracting basin at the origin.

\begin{figure}[htp]
	\includegraphics[width=2.in]{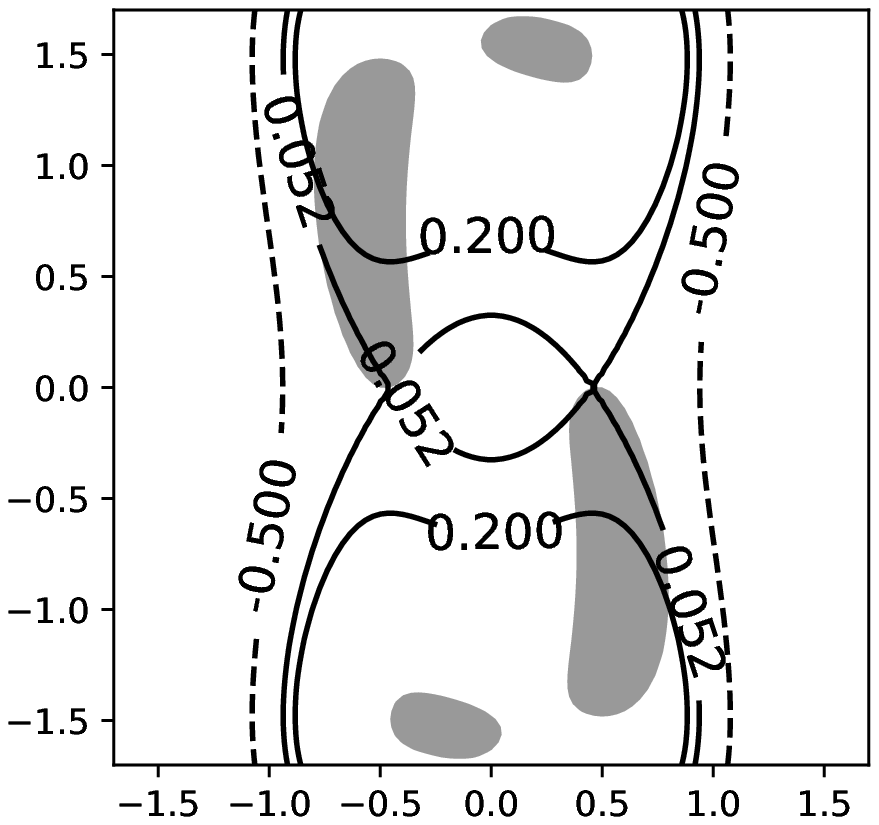}
	\includegraphics[width=2.in]{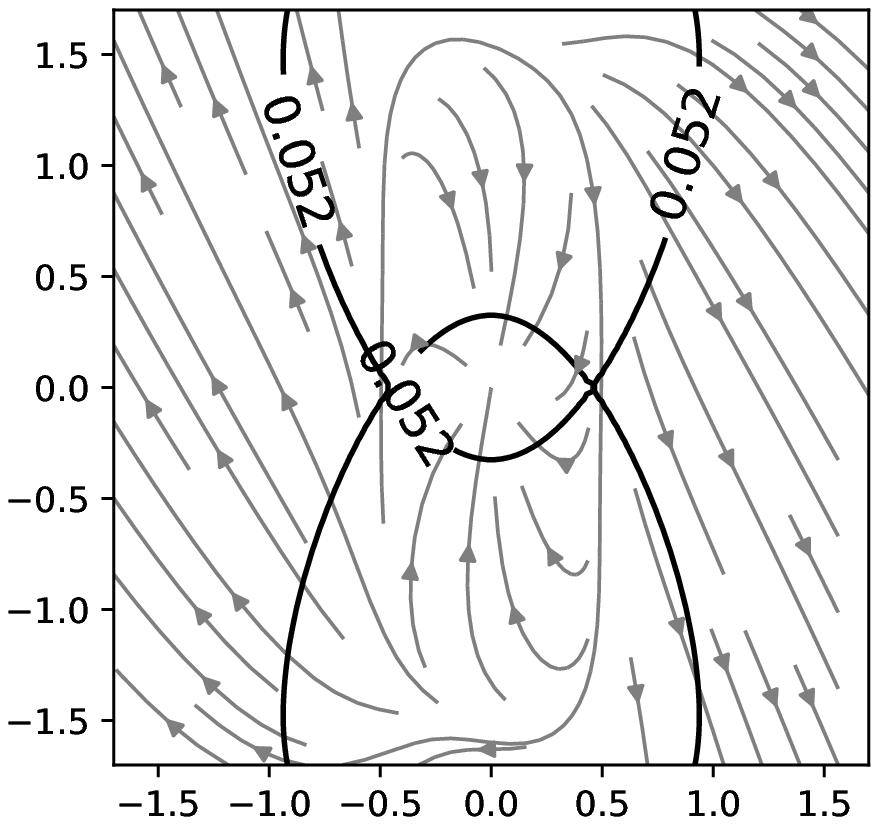}\\
		\caption{Left:  Contours of $V_1$ and the sign of $\dot{V_1}$. In the shaded domain, $\dot{V_1}$ is positive. Right: Solution curves of Equation (\ref{Giesl_ex}). A contour of $V_1$ is given for comparison with the left panel.}
\label{giesl_lya}
\end{figure}
 By retaking HHD, it may be possible to improve it. Another Lyapunov function is now constructed by adding harmonic polynomials to $V_1$. As the original vector field is roughly symmetric about the $y$ axis, the addition of, for example, the quartic harmonic functions of the form $ a \left( x^4 - 6x^2  y^2 +y^4\right)$ is considered, where $a$ is a constant. If $a=0.5$, we obtain a Lyapunov function $V_2(x,y) = V_1(x,y) + 0.5 \left( x^4 - 6x^2  y^2 +y^4\right)$. Contours of $V_1$ and the sign of $\dot{V_1}$ are shown in Figure \ref{giesl_lya_2}. It should be noted that the estimate for the local attracting basin at the origin is improved.

Only linear approximation is essentially required in the construction above. However, the resulting Lyapunov functions respect the nonlinear nature of the original equation, despite the ease of calculation. This is due to the fact that HHD reflects the nonlinearity of the original vector field. Furthermore, as numerous different Lyapunov functions may be obtained at a relatively low computational cost, it appears possible that a better picture of the local attracting basin could be obtained by comparing them. This points out another method for the computational analysis of the attracting basin.

\begin{figure}[htp]
		\includegraphics[width=2.in]{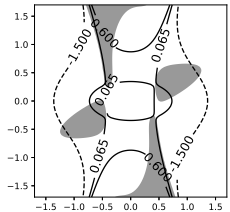}
		\caption{Contours of $V_2$ and the sign of $\dot{V_2}$. In the shaded domain, $\dot{V_2}$ is positive.}
\label{giesl_lya_2}
\end{figure}
\end{Ex}
\section{Strictly orthogonal HHD}
As was mentioned in Introduction, orthogonality in the $L^2$ sense is frequently imposed on HHD. Indeed, such a decomposition is obtained by imposing  as a boundary condition that the divergence-free vector field should be orthogonal to the normal vector of the boundary, and HHD with this boundary condition is unique \cite{Denaro_On_2003}. However, this is not sufficient if the aim is to completely analyze the behavior of vector fields. In particular, equilibrium points are not necessarily respected under this condition, and potential functions may have critical points irrelevant to the flow of the original vector field. These properties result in artifacts, as seen, for example, in \cite{Wiebel_Feature}.

In certain cases, a strictly orthogonal HHD may be obtained, which is sufficient for completely describing the behavior of the vector field.
	\begin{Def}
	The HHD ${\bf F} =- \nabla V+ {\bf u} $ is said to be {\bf strictly orthogonal} on $D \subset {\mathbb R}^n$ if ${\bf u}({\bf x}) \cdot  \nabla V({\bf x})= 0$ for all ${\bf x}\in D$.
	\end{Def}
If an HHD is strictly orthogonal, then it is orthogonal in the $L^2$ sense. For example, some vector fields given by relatively simple polynomials obviously have such a decomposition.
\begin{Ex}[Strictly orthogonal HHD]\label{ohhd_ex}
\normalfont
For $\mu >0 $ and $\omega < 0$, the following differential equation is considered:
	\begin{equation}\label{Hopf_eq}
\frac{d}{dt}\left(\begin{array}{c}
										x\\
										y
									\end{array}
								\right)
=\left(\begin{array}{c}
										\mu x -\omega y  - \left( x^2 + y^2\right)x\\
										\omega x + \mu y  - \left( x^2 + y^2\right)y
									\end{array}
								\right).
\end{equation}
The vector field on the right-hand side has a strictly orthogonal HHD, namely,
$$
V =  -\frac{\mu}{2} (x^2 +y^2) + \frac{1}{4} (x^2 +y^2)^2,
$$
$$
{\bf u} = \left(\begin{array}{c}
										-\omega y\\
										\omega x
									\end{array}
								\right).
$$
Solution curves of the vector field are illustrated in Figure \ref{Hopf_sol}. Solution curves of $-\nabla V$ and ${\bf u}$ given above are shown in Figure \ref{Hopf_HHD}.
\begin{figure}[htp]
\begin{center}
	\includegraphics[width=2.in]{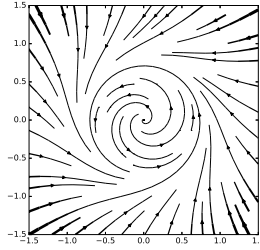}\\
	\caption{Solution curves of the vector field (\ref{Hopf_eq}). }
\end{center}
\label{Hopf_sol}\vspace*{-10pt}
\end{figure}
\begin{figure}[htp]
	\includegraphics[width=2.in]{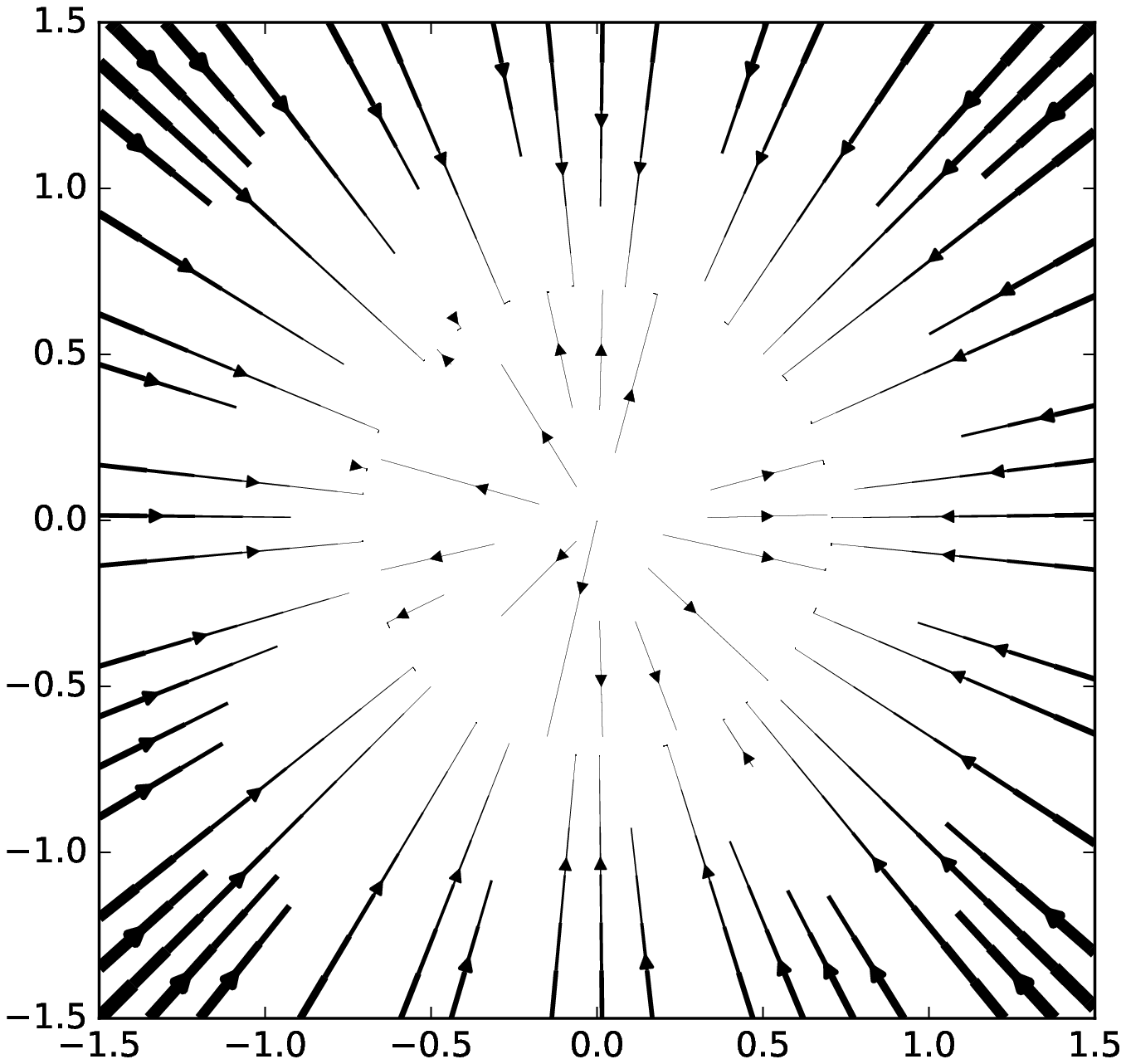}
	\includegraphics[width=2.in]{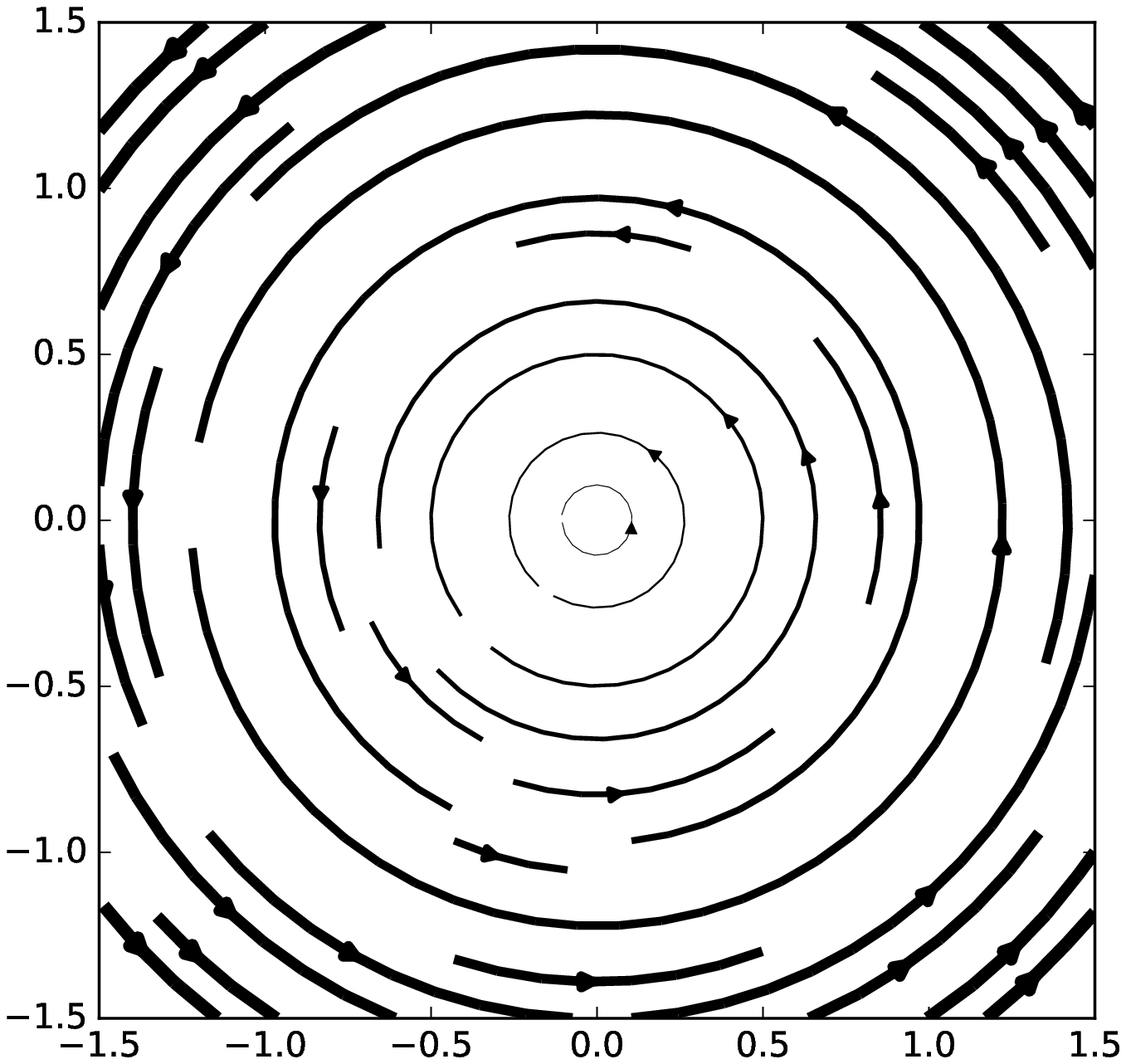}\\
		\caption{Strictly orthogonal HHD of the vector field (\ref{Hopf_eq}). Left: solution curves of $-\nabla V$. Right: solution curves of ${\bf u}$.}
\label{Hopf_HHD}
\end{figure}
\end{Ex}
For vector fields with strictly orthogonal HHD, there is a complete analogue to gradient systems, which are the extreme case of vector fields with Lyapunov functions. It is well known that the behavior of gradient vector fields is given in terms of potential functions \cite{ robinson1999dynamical}. The next theorem states that the same holds true if there is a strictly orthogonal HHD; therefore, vector fields with such a decomposition are a generalization of gradient vector fields.

\begin{Thm}[Main Theorem \ref{ohhd}]
	Let $D \subset {\mathbb R}^n$ be a bounded domain. If a vector field ${\bf F}:{\mathbb R}^n \to {\mathbb R}^n$ has a strictly orthogonal HHD on $\bar D$ with potential function $V$, then the following hold:
		\begin{enumerate}
		\item  If $D$ is forward invariant, then for all ${\bf x} \in D$, $$\omega({\bf x}) \subset \{{\bf y} \in {\bar D}\,| \,\nabla V({\bf y}) = {\bf 0}\}.$$
		\item If $D$ is backward invariant, then for all ${\bf x} \in D$,$$\alpha({\bf x}) \subset \{{\bf y} \in {\bar D}\,| \,\nabla V({\bf y}) = {\bf 0}\}.$$
		\end{enumerate}
	\end{Thm}

 The proof of this theorem is similar to that of LaSalle's invariance principle \cite{lasalle}. The following lemma is used in the proof of Main Theorem \ref{ohhd}.
	\begin{Lem}\label{ohhd_lem}
	Let $D \subset {\mathbb R}^n$ be a bounded domain. It is assumed that a vector field ${\bf F}:{\mathbb R}^n \to {\mathbb R}^n$ has a strictly orthogonal HHD on $\bar D$ with potential function $V$ and $D$ is forward invariant. For each point ${\bf x}\in D$, let $V^*\left({\bf x}, \cdot \right): [0, \infty ) \rightarrow {\mathbb R}$ be defined by
	$$
	V^*\left({\bf x}, T \right) := \sup\{V({\bf y}) \mid {\bf y}\in \overline{ \cup_{T \leq t} \left \{\phi^t ({\bf x})\right \}}\}.
	$$
Then, the following hold:
	\begin{enumerate}
		\item $V^*\left({\bf x}, \cdot \right)$ is well-defined and given explicitly by $$V^*\left({\bf x}, T \right) = V\left( \phi^T({\bf x}) \right).$$ In particular, $V^*\left({\bf x}, \cdot \right)$ is monotonically decreasing.
		\item For all ${\bf y}\in \omega({\bf x})$, we have $$V({\bf y}) = \lim_{T \rightarrow \infty} V^*\left({\bf x}, T \right) .$$
	\end{enumerate}
	\end{Lem}
		\begin{proof}
		(1) As $\dot{V} \leq 0$ for all $ {\bf y}\in \bigcup_{T \leq t} \left \{\phi^t ({\bf x})\right \}$, we have
		$$V({\bf y}) \leq V(\phi^T({\bf x})).$$
		For all $ \tilde{\bf y}\in \overline{\bigcup_{T \leq t} \left \{\phi^t ({\bf x})\right \}}$, there is a sequence ${\bf y}_{n}$ such that
		\begin{eqnarray*}
			\tilde{\bf y} &=& \lim_{n \to \infty}{\bf y}_{n},\\
			{\bf y}_{n} &\in& \bigcup_{T \leq t}\left \{\phi^t ({\bf x}) \right \}.
		\end{eqnarray*}
		The continuity of $V$ implies that
		\begin{eqnarray*}
			V(\tilde{\bf y}) &=& \lim_{n \to \infty} V({\bf y}_{n})\\
					&\leq& V\left(\phi^T({\bf x})\right).
		\end{eqnarray*}
		Therefore,
		$$V^*\left({\bf x}, T \right) \leq V(\phi^T({\bf x})).$$
		Moreover,
		$$ \phi^T({\bf x})\in \overline{ \bigcup_{T \leq t} \left \{\phi^t ({\bf x})\right \}};$$
		thus,
		$V^*\left({\bf x}, T \right) \geq V(\phi^T({\bf x}))$. Therefore, $V^*\left({\bf x}, T \right) = V\left( \phi^T({\bf x}) \right)$. Monotonicity is obvious.
		\\
		(2) The compactness of $\bar D$ implies that $V$ is bounded on $D$. Therefore, the monotonic function $V^*\left({\bf x}, \cdot \right)$ is also bounded. Thus, the quantity
		$$V^*\left({\bf x}, \infty \right) := \lim_{T \to \infty} V^*\left({\bf x}, T \right)$$
		is well-defined. If ${\bf y} \in \omega({\bf x})$, there exists a sequence $t_n$ such that
		\begin{eqnarray*}
			{\bf y} &=& \lim_{n \to \infty}\phi^{t_{n}}({\bf x}),\\
			\lim_{n \to \infty}t_n &=& \infty.
			\end{eqnarray*}
	Using the continuity of $V$ and (1), we conclude that
	\begin{eqnarray*}
		V({\bf y}) &=& \lim_{n \to \infty} V(\phi^{t_n}({\bf x}))\\
				&=& \lim_{n \to \infty}V^*\left({\bf x}, t_n \right)\\
				&=& V^*\left({\bf x}, \infty \right).
	\end{eqnarray*}
	Therefore, $V$ is constant on $\omega({\bf x})$ and equal to $V^*\left({\bf x}, \infty \right)$.
	\end{proof}
	 Main Theorem \ref{ohhd} is proved below.
	\begin{proof}
	
	The results for $\alpha$-limit sets follow from those for $\omega$-limit sets. Thus, only the latter case is proved here.
	
	Let ${\bf y} \in \omega({\bf x})$ be fixed. As $\omega({\bf x})$ is invariant for all $t \in {\mathbb R}$, we have
	$$\phi^t({\bf y}) \in \omega({\bf x}).$$
	By Lemma \ref{ohhd_lem}, we have for all $t \in {\mathbb R}$,
	$$V(\phi^t({\bf y})) = V^*\left({\bf x}, \infty \right).$$
	As
	\begin{eqnarray*}
		0 &=& \left.\frac{d}{d t} V^*\left({\bf x}, \infty \right)\right|_{t=0}\\
			&=& \left.\frac{d}{d t} V(\phi^t({\bf y}))\right|_{t=0}\\
			&=&\dot V ({\bf y})\\
			&=& - |\nabla V({\bf y})|^2,
	\end{eqnarray*}
	we conclude that $\nabla V({\bf y}) = {\bf 0}$. Therefore, $\omega({\bf x}) \subset \{{\bf y} \in {\bar D}\mid \nabla V({\bf y}) = {\bf 0}\}$.
	\end{proof}
\begin{Rem}
The method of Demongeot, Glade, and Forest is essentially an approximation by polynomial vector fields with a strictly orthogonal  HHD\cite{Demongeot_Li_2007_2}.
\end{Rem}
 The main limitation of Main Theorem \ref{ohhd} is the difficulty in finding a strictly orthogonal HHD. In general, its existence cannot be ensured. Furthermore, this theorem implies that the limit sets for a vector field with a strictly orthogonal HHD are restricted in smooth manifolds, which is not plausible for higher dimensional vector fields.

Nevertheless, if a strictly orthogonal HHD can be obtained, the behavior of the vector field can be rigorously determined, as the next example illustrates.
\begin{Ex}
\normalfont
As an application of Main Theorem \ref{ohhd}, it is shown that the vector field in Example \ref{ohhd_ex} has a limit cycle. As is easily verified, the vector field has a single equilibrium point at the origin. The set $S = \{{\bf y} \in {\mathbb R}^2 \mid \nabla V({\bf y}) = {\bf 0}\}$ is the union of the origin and the circle $C= \{{\bf y} \in {\mathbb R}^2 \mid {\bf y}\cdot {\bf y} = \mu \}$. By Main Theorem \ref{ohhd}, $\omega({\bf x}) \subset S$ for all ${\bf x} \in {\mathbb R}^2$. As $\omega({\bf x})$ is connected, we either have $\omega({\bf x}) =\{{\bf  0}\}$ or $\omega({\bf x})\subset C$. If ${\bf x} \neq {\bf 0}$, the former is impossible because the origin is unstable. Thus, the only possibility is that $\omega({\bf x})\subset C$. By the Poincar\'e-Bendixon theorem, $\omega({\bf x})$ is a periodic orbit; therefore, $\omega({\bf x}) = C$.
\end{Ex}
The existence of a strictly orthogonal HHD is compatible with a boundary condition commonly used if the domain is properly chosen, even though boundary conditions alone are not sufficient.
		\begin{Lem}
	Let ${\bf F}$ be a vector field defined on ${\mathbb R}^2$. Let $D \subset {\mathbb R}^2$ be a domain and $\partial D$ be a piecewise $C^1$ Jordan curve. It is assumed that all ${\bf x}\in \partial D$ are equilibrium points. If there is a strictly orthogonal HHD of ${\bf F}$ on $\bar D$, it satisfies the following on $\partial D$:
	$${\bf u}\cdot{\bf n} = 0,$$
			where ${\bf n}$ is normal vector of $\partial D$ pointing outward.
	\end{Lem}
		\begin{proof}
		We have ${\bf u}({\bf y}) ={\bf F}({\bf y}) = {\bf 0}$ on $\partial D$. Therefore,  ${\bf u}\cdot{\bf n} = 0$.
	\end{proof}
\begin{Lem}\label{bc_2_th}
Let ${\bf F}$ be a vector field defined on ${\mathbb R}^2$. Let $D \subset {\mathbb R}^2$ be a domain and $\partial D$ be a piecewise $C^1$ Jordan curve. It is assumed that there is a point ${\bf x}\in D$ such that $\omega({\bf x}) = \partial D$. If there is a strictly orthogonal HHD of ${\bf F}$ on $\bar D$, it satisfies the following on $\partial D$:
	$${\bf u}\cdot{\bf n} = 0,$$
			where ${\bf n}$ is normal vector of $\partial D$ pointing outward.
	\end{Lem}
		\begin{proof}
		Main Theorem \ref{ohhd} implies that $\nabla V = {\bf 0}$ on $\partial D$. Therefore, ${\bf u}({\bf y}) ={\bf F}({\bf y})$ on $\partial D$ and ${\bf u}$ is tangent to $\partial D$. Thus, ${\bf u}\cdot{\bf n} = 0$.
	\end{proof}
\noindent Under the assumptions in Theorem \ref{bc_2_th}, a strictly orthogonal HHD is unique.
		\begin{Cor}\label{Unique_so}
	Let ${\bf F}$ be a vector field defined on ${\mathbb R}^2$. Let $D \subset {\mathbb R}^2$ be a domain and $\partial D$ be a partially $C^1$ Jordan closed curve. It is assumed that there is a point ${\bf x}\in D$ such that $\omega({\bf x}) = \partial D$. Then, a strictly orthogonal HHD is unique if it exists.
	\end{Cor}
		\begin{proof}
		The HHD that satisfies ${\bf u}\cdot{\bf n} = 0$ on $\partial D$ is unique (a proof may be found in \cite{Denaro_On_2003}).
	\end{proof}
Uniqueness is useful for constructing strictly orthogonal HHD. In Corollary \ref{Unique_so}, the domain $D$ under consideration is determined in terms of the dynamics of the original vector field and there is no need for optimization. If a strictly orthogonal HHD exists, it will be constructed by applying the boundary condition ${\bf u}\cdot{\bf n} = 0$ on $\partial D$, owing to uniqueness.

\section{Analysis of planar vector fields}\label{planar}
In this section, to study the limitations of the analysis based on the potential functions of HHDs, planar vector fields are considered. In particular, the effect of boundary conditions on the properties of potential functions are studied. The results obtained here indicate that if the derivative of the vector field is degenerate at the equilibrium point, a Lyapunov function cannot be easily constructed by choosing a potential function.

For planar vector fields with isolated equilibrium points, the Poisson equation can be explicitly solved in neighborhoods of such points, and a formula for representing the derivatives of the potential function of the HHD can be obtained in terms of the Fourier coefficients of the boundary condition. Here,  the following property is used: the addition of a harmonic function to the potential function of an HHD yields another HHD. For simplicity, in what follows, it is assumed that the vector field has an equilibrium point at the origin. The Poisson kernel and its orbital derivative are first defined.
\begin{Def}
The Poisson kernel on the unit disk is defined by
$$K (x,y,\theta):= \frac{1-x^2-y^2}{(x-\cos(\theta))^2 + (y-\sin(\theta))^2} ,$$
where $\theta \in [0, 2 \pi)$ is a parameter.
For a vector field ${\bf F} := (f_1,f_2)$ on the unit disk, let
$$L[{\bf F}](x,y,\theta):= f_1\frac{\partial K}{\partial x}+ f_2\frac{\partial K}{\partial y}.$$
\end{Def}
In terms of the vector field, $L[{\bf F}]$ is given explicitly as follows.
\begin{Lem}
For a vector field ${\bf F}$ on the unit disk, $L[{\bf F}](x,y,\theta)$ is explicitly given by
 $$ \frac{2 \left( -(1+K){\bf x}+K{\bf e}_{\theta}\right) }{(x-\cos(\theta))^2+ (y-\sin(\theta))^2} \cdot {\bf F} ,$$
where ${\bf x} = (x,y)$ and ${\bf e}_{\theta} = \left(\cos(\theta), \sin(\theta) \right)$.
\end{Lem}
Representing a harmonic function on the unit disk using the Poisson kernel, we obtain the following result by interchanging differentiation and integration.
\begin{Lem}\label{rep_h}
If a harmonic function on the unit disk is given by
$$ h(x,y) = \frac{1}{ 2 \pi } \int_0^{2 \pi}\alpha(\theta) K(x,y,\theta) d \theta,$$
where $\alpha$ is a bounded function on the unit circle $S^1$,
then
$$ \dot{h}(x,y)= \frac{1}{ 2 \pi } \int_0^{2 \pi}\alpha(\theta) L[{\bf F}](x,y,\theta) d \theta.$$
\end{Lem}
The addition of harmonic functions changes the choice of the decomposition. If an HHD is given by ${\bf F} = - \nabla V +{\bf u}$, the addition of a harmonic function $h$ yields another decomposition of the form ${\bf F} = - \nabla\left( V +h\right)+\left({\bf u}+\nabla h \right)$. Therefore, it suffices to determine a harmonic function $h$ so that $V+h$ is a Lyapunov function of the origin, where $V$ is the potential function of the initial decomposition. The aim here is to calculate the value of the partial differentials of $V+h$ and ${\dot V}+{\dot h}$ at the origin using  Lemma \ref{rep_h} and thereby establish that $V+h$ attains a minimum, whereas ${\dot V}+{\dot h}$ attains a maximum at the origin. The following result can be shown by direct calculation.
\begin{Lem}\label{rep_K}
The values of the partial differentials of $K$ at the origin are
$$\frac{\partial K}{\partial x}(0,0,\theta) = 2\cos(\theta).$$
$$\frac{\partial K}{\partial y}(0,0,\theta) = 2 \sin(\theta).$$
$$\frac{\partial^2 K}{\partial x^2}(0,0,\theta) = 4 \cos(2 \theta).$$
$$\frac{\partial^2 K}{\partial x \partial y}(0,0,\theta) =  4 \sin(2 \theta).$$
$$\frac{\partial^2 K}{\partial y^2}(0,0,\theta) = -4 \cos(2 \theta).$$

\end{Lem}
By Lemma \ref{rep_K}, we obtain formulas for the values of the derivatives of a harmonic function at the origin by interchanging differentiation and integration.
\begin{Prop}\label{fc_h}
Let a harmonic function on the unit disk be given by
$$ h(x,y) = \frac{1}{ 2 \pi } \int_0^{2 \pi}\alpha(\theta) K(x,y,\theta) d \theta,$$
where $\alpha$ has a Fourier series expansion $$\alpha(\theta) = a_0 + \sum_{k = 0}^{\infty} \left(a_k\cos(k \theta) + b_k \sin(k \theta)\right).$$
Then, the values of the partial differentials of $h$ at the origin are
$$\frac{\partial h}{\partial x}(0,0) =2 a_1.$$
$$\frac{\partial h}{\partial y}(0,0) =2 b_1.$$
$$\frac{\partial^2 h}{\partial x^2}(0,0) =  4 a_2.$$
$$ \frac{\partial^2 h }{\partial x \partial y}(0,0) = 4 b_2.$$
$$\frac{\partial^2 h}{\partial y^2}(0,0) =  -4 a_2.$$
\end{Prop}
Combining the results above, we obtain a sufficient condition for the Fourier coefficients whereby $V+h$ attains a minimum at the origin.
\begin{Prop}
Let $V$ be a potential function of an HHD and $h$ be a harmonic function that satisfies the condition in Proposition \ref{fc_h}. If $V$ assumes a minimum at the origin and $a_1 = b_1=0$, then the origin is a critical point of $V+h$.  Furthermore, if $V$ is the potential function constructed in Main Theorem \ref{Lya}, then the condition
\begin{equation}\label{f_c}
0 \leq a_2 ^2 + b_2 ^2 < \frac{\left( {\rm tr} D {\bf F} _0 \right)^2}{64}
\end{equation}
 ensures that $V+h$ also assumes a minimum at ${\bf 0}$.
\end{Prop}
\begin{proof}
By Proposition \ref{fc_h}, we immediately obtain that $\nabla ( V + h) ({\bf 0})= {\bf 0} $ if $a_1 = b_1=0$. Let V be the potential function constructed in Main Theorem \ref{Lya}. The Hessian of $V+h$ at the origin is
$$
\left(\begin{array}{cc}
	\frac{\rho}{2} + 4 a_2 & 4 b_2\\
	4 b_2 &\frac{\rho}{2} -4 a_2
	\end{array}
\right),
$$
where $\rho =  {\rm tr} D {\bf F} _0 ={\rm tr}\, {\rm Hess}\, V$.
As the trace of this matrix is positive, it is positive definite if its determinant is positive. This is equivalent to $\frac{\rho^2}{4} - 16 \left( a_2^2 + b_2^2 \right) > 0$.
\end{proof}
 ${\dot V}+{\dot h}$ is now considered. By calculations similar to those above, the following results are obtained.
\begin{Lem}
The values of the partial differentials of $L[{\bf F}]$ at the origin are
$$L[{\bf F}](0,0,\theta) = 0.$$
$$\frac{\partial L[{\bf F}]}{\partial x}(0,0,\theta) = 2\left(\cos(\theta) \frac{\partial f_1}{\partial x} + \sin(\theta) \frac{\partial f_2}{\partial x}\right).$$
$$\frac{\partial L[{\bf F}]}{\partial y}(0,0,\theta) = 2\left(\cos(\theta) \frac{\partial f_1}{\partial y} + \sin(\theta) \frac{\partial f_2}{\partial y}\right).$$
$$\frac{\partial^2 L[{\bf F}]}{\partial x^2}(0,0,\theta) = 8\left(\cos(2 \theta) \frac{\partial f_1}{\partial x} + \sin(2 \theta) \frac{\partial f_2}{\partial x}\right) + 2\left(\cos(\theta)  \frac{\partial^2 f_1}{\partial x^2} + \sin(\theta) \frac{\partial^2 f_2}{\partial x^2}\right).$$
$$\frac{\partial^2 L[{\bf F}]}{\partial x \partial y}(0,0,\theta) = $$
$$4\cos(2 \theta)\left(\frac{\partial f_1}{\partial y} - \frac{\partial f_2}{\partial x}\right) + 4 \sin(2 \theta)\left( \frac{\partial f_1}{\partial x} + \frac{\partial f_2}{\partial y}\right)+ 2\left(\cos(\theta)  \frac{\partial^2 f_1}{\partial x \partial y} + \sin(\theta) \frac{\partial^2 f_2}{\partial x \partial y}\right).$$
$$\frac{\partial^2 L[{\bf F}]}{\partial y^2}(0,0,\theta) =  8 \left(\cos(2 \theta) \frac{\partial f_1}{\partial y} + \sin(2 \theta) \frac{\partial f_2}{\partial y}\right) + 2\left(\cos(\theta)  \frac{\partial^2 f_1}{\partial y^2} + \sin(\theta) \frac{\partial^2 f_2}{\partial y^2}\right).$$
\end{Lem}

\begin{Prop}\label{rep_dot}
Let a harmonic function on the unit disk be given by
$$ h(x,y) = \frac{1}{ 2 \pi } \int_0^{2 \pi}\alpha(\theta) K(x,y,\theta) d \theta,$$
where $\alpha$ has a Fourier series expansion $$\alpha(\theta) = a_0 + \sum_{k = 0}^{\infty} \left(a_k\cos(k \theta) + b_k \sin(k \theta)\right).$$
Then, the values of partial differentials of ${\dot h}$ at the origin are
$${\dot h}(0,0) = 0.$$
$$\frac{\partial \dot{h}}{\partial x}(0,0) = 2\left(a_1 \frac{\partial f_1}{\partial x} + b_1 \frac{\partial f_2}{\partial x}\right).$$
$$\frac{\partial \dot{h}}{\partial y}(0,0) = 2\left(a_1 \frac{\partial f_1}{\partial y} + b_1 \frac{\partial f_2}{\partial y}\right).$$
$$\frac{\partial^2 \dot{h}}{\partial x^2}(0,0) =  8 \left( a_2 \frac{\partial f_1}{\partial x} + b_2 \frac{\partial f_2}{\partial x}\right) + 2\left(a_1  \frac{\partial^2 f_1}{\partial x^2} + b_1 \frac{\partial^2 f_2}{\partial x^2}\right).$$
$$ \frac{\partial^2 \dot{h} }{\partial x \partial y}(0,0) = 4 a_2 \left( \frac{\partial f_1}{\partial y} -\frac{\partial f_2}{\partial x}\right)+ 4 b_2 \left( \frac{\partial f_1}{\partial x} + \frac{\partial f_2}{\partial y}\right)+ 2\left(a_1 \frac{\partial^2 f_1}{\partial x \partial y} + b_1 \frac{\partial^2 f_2}{\partial x \partial y}\right).$$
$$\frac{\partial^2 \dot{h}}{\partial y^2}(0,0) =  8 \left( a_2 \frac{\partial f_1}{\partial y} + b_2 \frac{\partial f_2}{\partial y}\right) + 2\left(a_1  \frac{\partial^2 f_1}{\partial y^2} + b_1 \frac{\partial^2 f_2}{\partial y^2}\right).$$
\end{Prop}
Thus, the problem of obtaining a decomposition that yields a Lyapunov function is reformulated as an optimization problem for Fourier coefficients. By Proposition \ref{rep_dot}, the following result is immediately obtained.
\begin{Prop}
Let ${\bf F}:{\mathbb R}^2 \to {\mathbb R}^2$ be a vector field with an equilibrium point at the origin, and let ${\bf F} = - \nabla V + {\bf u}$ be an HHD and $h$ be a harmonic function that satisfies the condition in Proposition \ref{fc_h}. If $V$ assumes a minimum at $0$ and $a_1 = b_1=0$, then $0$ is a critical point of $\dot{V}+\dot{h}$.
\end{Prop}
\begin{proof}
From Proposition \ref{min_p}, we have $\nabla \dot V ({\bf 0})= 0$. Therefore, it suffices to show that $\nabla \dot h ({\bf 0}) = 0$, but this is obvious from Proposition \ref{rep_dot}.
\end{proof}

To obtain a Lyapunov function of the origin, it suffices that $ \dot{V}+\dot{h} \leq 0$. Thus, the Fourier coefficients $a_2, b_2$ should satisfy the condition (\ref{f_c}) so that the Hessian of $\dot{V}+\dot{h}$ is negative definite. However, this is not always possible because in the calculation above, only the first-order derivatives were considered. For example, the negative definiteness of  $\dot{V}+\dot{h}$ cannot be ensured in the case of $D {\bf F}_{\bf 0} = 0$. This is a limitation of the construction of Lyapunov functions based on the potential functions of HHDs. In general, higher-order derivatives and, therefore, Fourier coefficients for larger $n$ should be considered for such equilibrium points.

\section{Concluding remarks}
Even though the results obtained in this study cannot be applied to ``real'' problems, the construction of Lyapunov functions using HHD was shown to be possible. There are two questions to be answered regarding the general construction of Lyapunov functions of vector fields using HHD:
\begin{enumerate}
\item To what extent is the Lyapunov function globally valid if it is constructed by the proposed method?
\item Under what conditions is it possible to construct Lyapunov functions using Main Theorem \ref{opti_thm}? In that case, how can it be efficiently constructed?
\item Is it possible to construct Lyapunov functions using HHD in the case of vector fields with degenerate Jacobian matrix?
\end{enumerate}
\noindent
The first question is important in the applications because a global Lyapunov function would yield information about the global behavior of solutions. However, the construction in this study does not ensure the validity outside a neighborhood of the equilibrium point.
The second question should be answered if the method proposed here is to be applied to real problems. A certain stability condition is clearly necessary; however, it is not obvious. Furthermore, an optimization problem should be solved to apply Main Theorem \ref{opti_thm}, and therefore an efficient scheme is required.
The third question concerns the superiority of the analysis based on HHD compared to the existing methods (e.g., linearization). An affirmative answer would imply that the stability of an equilibrium point could be determined for a wider class of vector fields.

Moreover, the properties of vector fields with strictly orthogonal HHD should be further studied. In particular, conditions ensuring its existence are particularly important, as such vector fields can be rigorously analyzed in detail. Another problem is to refine Main Theorem \ref{ohhd} so that it provides a characterization of $\alpha$- or $\omega$-limit sets.

For the problem of choosing an HHD that respects the dynamics of vector fields, a decomposition that yields a Lyapunov function is a potential solution. However, this requires further investigation.
\section*{Acknowledgment}
I would like to express my gratitude to Professor Masashi Kisaka for his patient guidance and useful critique of this research. This study was supported by Grant-in-Aid for JSPS Fellows (17J03931). I would like to thank Editage (www.editage.jp) for English language editing. Further, I would like to thank the anonymous reviewers for their valuable comments and suggestions.
\providecommand{\href}[2]{#2}
\providecommand{\arxiv}[1]{\href{http://arxiv.org/abs/#1}{arXiv:#1}}
\providecommand{\url}[1]{\texttt{#1}}
\providecommand{\urlprefix}{URL }

\medskip
Received November  2017; 1st revision  January 2018; final revision  May 2018.
\medskip


\begin{thebibliography}{10}
\bibitem{Bhatia_The_2013}
\newblock H.~Bhatia, G.~Norgard and V.~Pascucci,
\newblock The {Helmholtz-Hodge} decomposition-a survey,
\newblock \emph{IEEE T. Vis. Comput. Gr.}, \textbf{19} (2013), 1386--1404.

\bibitem{Bhatia_The_2014} [10.1109/TVCG.2014.2312012]
\newblock H.~Bhatia, V.~Pascucci and P.~Bremer,
\newblock \doititle{The natural {Helmholtz-Hodge} decomposition for {Open-Boundary} flow analysis},
\newblock \emph{IEEE T. Vis. Comput. Gr.}, \textbf{20} (2014), 1566--1578.

\bibitem{Demongeot_Li_2007} (MR2288602) [10.1016/j.crma.2006.10.016]
\newblock J.~Demongeot, N.~Glade and L.~Forest,
\newblock \doititle{Li{\'e}nard systems and {potential-Hamiltonian} decomposition {I} -methodology},
\newblock \emph{C. R. Math.}, \textbf{344} (2007), 121--126.

\bibitem{Demongeot_Li_2007_2} (MR2292286) [10.1016/j.crma.2006.10.013]
\newblock J.~Demongeot, N.~Glade and L.~Forest,
\newblock \doititle{Li{\'e}nard systems and {potential-Hamiltonian} decomposition {II} -algorithm},
\newblock \emph{C. R. Math.}, \textbf{344} (2007), 191--194.

\bibitem{Denaro_On_2003} (MR1998824) [10.1002/fld.598]
\newblock M.~Denaro,
\newblock \doititle{On the application of the {Helmholtz-Hodge} decomposition in projection methods for incompressible flows with general boundary conditions},
\newblock \emph{Int. J. Numer. Methods Fluids}, \textbf{43} (2003), 43--69.

\bibitem{Duarte_Deformation_1983} (MR709510) [10.1063/1.525894]
\newblock T.~Duarte and R.~Mendes,
\newblock \doititle{Deformation of {Hamiltonian} dynamics and constants of motion in dissipative systems},
\newblock \emph{J. Math. Phys.}, \textbf{24} (1983), 1772--1778.

\bibitem{Fuselier_A_2016} (MR3649426) [10.1093/imanum/drw027]
\newblock E.~Fuselier and G.~Wright,
\newblock \doititle{A radial basis function method for computing {Helmholtz-Hodge} decompositions},
\newblock \emph{IMA J. Numer. Anal.}, \textbf{37} (2017), 774--797.

\bibitem{Geng_The_2010} (MR2671125) [10.1016/j.jfa.2010.07.005]
\newblock J.~Geng and Z.~Shen,
\newblock \doititle{The {Neumann} problem and {Helmholtz} decomposition in convex domains},
\newblock \emph{J. Funct. Anal.}, \textbf{259} (2010), 2147--2164.

\bibitem{Giesl_Review_2015} (MR3423237) [10.3934/dcdsb.2015.20.2291]
\newblock P.~Giesl and S.~Hafstein,
\newblock \doititle{Review on computational methods for {Lyapunov} functions},
\newblock \emph{Discrete Contin. Dyn. Syst. Ser. B}, \textbf{20} (2015), 2291--2331.

\bibitem{giesl2007construction} (MR2313542)
\newblock P.~Giesl,
\newblock \emph{Construction of Global Lyapunov Functions Using Radial Basis Functions},
\newblock Springer--Verlag Berlin Heidelberg, 2007.

\bibitem{lasalle} (MR0118902)
\newblock J.~LaSalle,
\newblock Some extensions of Liapunov's second method,
\newblock \emph{IRE Trans. Circuit Theory}, \textbf{7} (1960), 520--527.

\bibitem{599986} (MR1469846) [10.1109/9.599986]
\newblock A.~Polanski,
\newblock \doititle{Lyapunov function construction by linear programming},
\newblock \emph{IEEE Trans. Autom. Control.}, \textbf{42} (1997), 1013--1016.

\bibitem{Polthier_Identifying_2003} (MR2047004)
\newblock K.~Polthier and E.~Preu{\ss},
\newblock Identifying vector field singularities using a discrete Hodge decomposition,
\newblock \emph{Visualization and Mathematics {III}}, 2003, 113--134.

\bibitem{robinson1999dynamical} (MR1792240)
\newblock C.~Robinson,
\newblock \emph{Dynamical Systems: Stability, Symbolic Dynamics, and Chaos},
\newblock Studies in Advanced Mathematics, CRC-Press, 1999.

\bibitem{strang2006linear}
\newblock G.~Strang,
\newblock \emph{Linear Algebra and Its Applications},
\newblock Thomson, Brooks/Cole, 2006.

\bibitem{Wiebel_Feature}
\newblock A.~Wiebel,
\newblock Feature detection in vector fields using the {Helmholtz-Hodge} decomposition {Diploma-Thesis}.
\newblock


\end{thebibliography}
\end{document}